\documentclass[12pt,reqno]{amsart}

\usepackage{amsmath, amssymb, amsfonts, amsbsy, amsthm, latexsym, epsfig}
\usepackage{dirtytalk}
\theoremstyle{definition}
\newtheorem{theorem}{Theorem}

\theoremstyle{definition}
\newtheorem{corollary}[theorem]{Corollary}
\newtheorem{lemma}[theorem]{Lemma}
\newtheorem{proposition}[theorem]{Proposition}
\newtheorem{definition}[theorem]{Definition}

\newtheorem*{theorem*}{Theorem}
\newtheorem*{corollary*}{Corollary}
\newtheorem*{proposition*}{Proposition}
\newtheorem*{lemma*}{Lemma}
\newtheorem*{remark*}{Remark}
\numberwithin{equation}{section}

\begin{document}

\title{Frames via Unilateral Iterations of a Bounded Operator}

\author{Victor Bailey}

\address{Department of Mathematics,
University of Oklahoma,
Norman, OK 73019 USA}
\email{victor.bailey@ou.edu}

\maketitle 
\begin{abstract}
Motivated by recent work in Dynamical Sampling, we prove a necessary and sufficient condition for a frame in a separable and infinite-dimensional Hilbert space to admit the form $\{T^{n} \varphi \}_{n \geq 0}$ with $T \in B(H)$. Also, a characterization of all the vectors $\varphi$ for which $\{T^{n} \varphi \}_{n \geq 0}$ is a frame for some $T \in B(H)$ is provided. Some auxiliary results on operator representations of Riesz frames are given as well.  %In particular, we show that a frame, $\{f_n\}_{n \in \mathbb{N}_0}$, admits the form $\{T^{n} \varphi \}_{n \geq 0}$ with $T \in B(H)$ if and only if the frame is a Riesz basis or it is linearly independent and the system, $\{S^nc\}_{n \in \mathbb{N}_{0}}$, is a Parseval frame for the kernel of the synthesis operator (where $S$ is the right-shift operator on $\ell^2(\mathbb{N}_{0})$ and $c$ is some fixed scalar sequence).
\end{abstract}
\section
{Introduction}

Frames were first introduced in the foundational paper \cite{SD52} by Duffin and Schaeffer. These are systems of vectors in a separable Hilbert Space, $H$, that admit basis-like expansions for each element of $H$. That is, for a system $\{f_n\}_{n \in I} \subset H $, with $I$ countable, that satisfies the so-called frame condition, each vector $f\in H$ can be given a series representation in terms of the frame elements so that $f = \underset{n \in I}{\sum} c_n \, f_n $
where $\{c_n\}_{n \in I} \in \ell^2(I)$ and convergence is in the norm of $H$. For a basis, expansions of this form are unique. However for a frame these expansions are not necessarily unique so that, in a sense, a frame can be viewed as a generalized basis for $H$. A system $\{f_n\}_{n \in I}$ is said to be a frame if there exist fixed constants $ 0 < C_{1} \leq C_{2}$ such that for each $ f \in H$,  \begin{displaymath} C_{1} \|f\|^2 \leq \sum_{n \in I} | \langle f , f_n \rangle |^2 \leq C_{2} \|f\|^2. \quad 
\end{displaymath}
\quad The Dynamical Sampling problem \cite{AA16} was inspired from a paper \cite{LV09} by Lu and Vetterli where they sought to reconstruct a vector using an insufficient spacial sampling density by employing an evolution operator to compensate for the reduced spacial sampling. In Dynamical Sampling, the goal is the same, however we no longer consider the domain of the vectors in $H$ when obtaining samples and instead seek to recover any vector $f \in H$ using the samples $\{{\langle A^n f, g\rangle}\}_{{0 \leq n \leq \Psi (g)} , \, g \, \in \, G}$
where $A \in B(H)$, $G$ is a countable subset of $H,$ and $\Psi$ is a map from $G$ to $\mathbb{N}_0 \cup \{\infty\}$ (where $\mathbb{N}_0 = \mathbb{N} \cup \{0\}$) . Since the samples given above are the same as $\{{\langle f, (A^*)^n g\rangle}\}_{{0 \leq n \leq \Psi (g)} , \, g \, \in \, G}$ we find that any vector $f \in H$ can be stably reconstructed from the samples above if and only if $\{(A^*)^n g\}_{{0 \leq n \leq \Psi (g)} , \, g \, \in \, G}$ is a frame for $H$. Due to this, we forgo the adjoint and seek conditions for an operator $T$, the set $G$, and the map $\Psi$ such that the system $\{T^n g\}_{{0 \leq n \leq \Psi (g)}, \, g \, \in \, G}$ is a frame or has properties such as completeness, exactness, etc.
\newline
%Many recent results have been found in the case where the set $G$ consists of a single vector. In this case, the Dynamical Sampling problem is to find conditions on a vector $\varphi$ and an operator $T \in B(H)$ such that $\{T^n\varphi \}_{n \geq 0}$ provides basis-like expansions of vectors in $H$ or is useful in some way for reconstructing vectors in $H$. 

A complete solution to the Dynamical Sampling problem in the finite-dimensional Hilbert space setting was provided in \cite{ACMT17}. Many recent works \cite{CH21} consider the problem in the infinite-dimensional setting and typically the problem is posed in the following way: Given an operator $T \in B(H)$ and a vector $\varphi \in H$, what are conditions for the system $\{T^n \varphi\}_{n \geq 0}$ to be a frame, basis, Bessel, complete, minimal, etc.? Here the set $G$ consists of only one vector so that infinitely many iterations of the operator will be needed to even obtain a complete set in $H$ when $H$ is infinite dimensional. In this version of the Dynamical Sampling problem we do not require $T$ to be invertible, but ask for conditions on the operator $T$ and vector $\varphi$ so that the system $\{T^n \varphi\}_{n \geq 0}$ is useful in some way (e.g. allowing basis-like expansions of arbitrary vectors in $H$) for reconstructing the Hilbert space $H$.
However, many of the results on whether a system of vectors that are given by iterations of some bounded operator on a fixed vector allow for expansions of arbitrary vectors in $H$ via that system, are negative results. By \cite{ACM16}  it holds that if $T \in B(H)$ is a normal operator, then $\{T^n \varphi\}_{n \geq 0}$ can never be a basis. Using the Spectral Theorem it is shown that a system of this form can never be both complete and minimal. This implies that such a system can never be a basis as all bases are exact systems. In \cite{{AA16}} it is shown that if $\{T^n \varphi\}_{n \geq 0}$ is a frame for $H$, then $(T^*)^n f \to 0$ as $n \to \infty$ so that given $T$ is unitary, then $\{T^n \varphi\}_{n \geq 0}$ can never be a frame. By \cite{CH21} it holds that if $T$ is compact, then $\{T^n \varphi\}_{n \geq 0}$ can never be a frame. Also, from \cite{CH21} it holds that if $T$ is a hypercyclic operator (that is, if there exists an $h \in H$ such that $\{T^n h \}_{n \geq 0}$ is dense in $H$), then $\{T^n \varphi\}_{n \geq 0}$ can never be a frame. In \cite{ACM16} by applying both  Feichtinger's Theorem and Müntz–Szász Theorem \cite{Hei11} it is shown that if $T$ is self-adjoint, then $\{ \frac{T^n \varphi}{ \|T^n \varphi \|}\}_{n \geq 0}$ can never be a frame. 
\newline

In light of this litany of negative results it seems to beg the question as to whether it is reasonable to expect that one can generate a frame or basis using iterations of an operator on a fixed vector. It seems to suggest that it would be useful to know properties that all frames of the form $\{T^n \varphi\}_{n \geq 0}$, for some bounded operator $T$, exhibit in order to determine conditions when these systems allow for expansions of arbitrary vectors in $H$. A problem related to Dynamical Sampling addresses this question. In particular, in recent works \cite{CH21, CH16} authors determine under what conditions can a given frame for $H$, $\{f_n\}_{n \in I}$, be represented in the form $\{T^n\varphi \}_{n \geq 0}$. That is, conditions are provided so that a given frame can be rewritten as a frame obtained by iterations of a bounded operator on a single vector. Applying results from this area, a solution to the Dynamical Sampling problem in this setting was provided in \cite{CHP20}. 
\newline

In this work we will make extensive use of the Hardy Space  \begin{displaymath}
   H^2(\mathbb{T}) = \{f \in L^2(\mathbb{T}) : \int_{\mathbb{T}} f(z) \overline{z}^n dz = 0 \, \,  \forall \, n<0 \}.
\end{displaymath}
That is the collection of square-integrable functions on the unit circle with vanishing negative Fourier coefficients. 
There is a natural identification between functions in 
\begin{displaymath}
    H^2(\mathbb{D}) = \{ f: \mathbb{D} \to \mathbb{C} \, \, \text{such that} \, \, f(z) = \underset{n \geq 0}{\sum} c_{n}z^n \, \, where \, \, \, (c_{n}) \in \ell^2(\mathbb{N}_0)\},
\end{displaymath} 
 the space of analytic functions on $\mathbb{D}$, the open unit disk in $\mathbb{C}$, with power series representations with square-summable complex coefficients, and those in $H^2(\mathbb{T})$. These Hilbert spaces are isomorphic and we can alternatively define $H^2(\mathbb{T})$ as  the radial limits of functions in $H^2(\mathbb{D})$. 
 \newline
 
An important unitary map between the sequence space $\ell^2(\mathbb{N}_0)$ and the Hardy-Hilbert Space is
\begin{displaymath}
A: \ell^2(\mathbb{N}_0) \to H^2(\mathbb{T}) 
\end{displaymath}
\begin{displaymath}
A(c_0, c_1, \ldots) = \underset{n \geq 0}{\sum} c_{n}z^n.
\end{displaymath}
The map $A$ is a unitary operator that satisfies $AR=SA$ where $R$ is the right-shift operator on $\ell^2(\mathbb{N}_0)$ and $S$ is the unilateral shift operator on $H^2(\mathbb{T})$ defined by  $(Sf)(z)=zf(z)$ for $f \in H^2(\mathbb{T})$.
\newline

There is a deep connection to the between Dynamical Sampling and the theory of Hardy Spaces which can be shown using a result from \cite{CH16}. In \cite{CH16} is shown that in order for a frame $\{f_k\}_{k \in I} \subset H$ to admit the form required in Dynamical Sampling, $\{T^n\varphi \}_{n \geq 0}$ where $T \in B(H)$, the kernel of the synthesis operator of the frame must be invariant under the right-shift operator on $\ell^2(\mathbb{N}_0)$. An important observation is that as right-shift invariant subspaces of $\ell^2(\mathbb{N}_0)$ correspond to invariant subspaces of the unilateral shift on $H^2(\mathbb{T})$ via the map $A$ given above. As shift invariant subspaces of $H^2(\mathbb{T})$ are all of Beurling-type \cite{B48}, this has important implications for the frames that are generated by iterations of a bounded operator on a single vector which we will see in the sequel. 
\newline
%Combining Theorem 15, Proposition 16, and Proposition 18 our main result is the statement immediately below.
%\begin{theorem*}
%Let $\{f_n\}_{n \in \mathbb{N}_0}$ be a frame for $H$. Then the following are equivalent:

%$(i.)$ $\{f_n\}_{n \in \mathbb{N}_0}$  admits the form $\{T^{n} \varphi \}_{n \geq 0}$ with $T \in B(H)$.
%\medskip

%$(ii.)$  $\{f_n\}_{n \in \mathbb{N}_0}$ is a Riesz basis or it is linearly independent and $\{S^{n} c\}_{n \geq 0}$ is a Parseval frame for $Ker(U)$, the kernel of the synthesis operator, for some $c \in \ell^2(\mathbb{N}_0)$ such that $A(c) \in H^2(\mathbb{T})$ is an inner function that is not a finite Blaschke product. 
%\end{theorem*}
We will also give additional results regarding operator orbit representations for Riesz frames, which are frames that satisfy the subframe property with uniform frame bounds. For details regarding of Riesz frames, we refer to \cite{CA02, Chr03}.

\section
{Results}

\begin{definition}
Given a frame $\{f_k\}_{k \in I} \subset H$ we define $U: l^2(I) \to H$ where $U \{c_k\}_{k} = \sum_{k} c_k \, f_k $ as the synthesis operator associated with the frame.
\end{definition}

If $C_2$ be the upper frame bound of a given frame, then we have  $\|Uc\|_{H} \leq \sqrt{C_2} \, \|c\|_{\ell^2(\mathbb{N}_0)}$ for all scalar sequences $c \in \ell^2(\mathbb{N}_0)$. Therefore, the synthesis operator, $U$, is a bounded linear operator.

\begin{definition}
The Lebesgue space 
\begin{displaymath}
L^2({\mathbb{T}}) = \{f: \mathbb{T} \to \mathbb{C} \, \, \text{such that} \, \,  \biggl(\int_{\mathbb{T}} |f(z)|^2 \, dz \biggl)^{1/2} < \infty\}
\end{displaymath} is a Hilbert space with inner product given by $\langle f, g \rangle = \int_{\mathbb{T}} f(z) \, \overline{g(z)} \, dz$.
\end{definition}

\begin{definition}
The Hardy Space
\begin{displaymath}
H^2(\mathbb{T})= \{f \in L^2(\mathbb{T}) : \int_{\mathbb{T}} f(z) \, \overline{z}^n \, dz = 0 , \, \forall n<0 \}
\end{displaymath} is a subspace of $L^2({\mathbb{T}})$ consisting of functions with vanishing negative Fourier coefficients.
\end{definition}

Note that if $f \in H^2(\mathbb{T})$, then $f(z) = \underset{n \geq 0}{\sum} c_{n} z^n$ and for any 
$g, h  \in H^2(\mathbb{T})$ with $h(z) = \underset{n \geq 0}{\sum} a_{n} z^n$ and $g(z) = \underset{n \geq 0}{\sum} b_{n} z^n$, 
   we have $\langle g, h \rangle =  \underset{n \geq 0}{\sum} a_{n}\overline{b_{n}}$.
   
\begin{definition}
A subspace of a Hilbert space, is a subset of $H$ that is closed in the topological sense in addition to being closed under the vector space operations.
\end{definition}
Here we have used the definition of subspace given in \cite{MR07}. Note that the kernel of any bounded linear operator is always a topologically closed vector space.
\begin{definition}
The right-shift operator on $\ell^2(\mathbb{N}_0)$ is the map $R: \ell^2(\mathbb{N}_0) \to \ell^2(\mathbb{N}_0)$ where $R (x_0, x_1, x_2 \ldots) = (0, x_0, x_1, x_2, \ldots)$.
\end{definition}

\begin{definition}
Given an operator $\Psi \in B(H)$ and a subspace $W \subset H$, if $\Psi (W) \subset W$, then $W$ is an invariant subspace for $\Psi$ (or invariant under $\Psi$). If both $W$ and $W^{\perp}$ are invariant subspaces for $\Psi$, then $W$ is a reducing subspace for $\Psi$.
\end{definition}
\begin{proposition}
If $\{f_n\}_{n \geq 0}$ is a frame such that $Ker(U)$ is a reducing subspace of the right-shift operator, then $\{f_n\}_{n \geq 0}$ is a Riesz basis.
\end{proposition}

\begin{proof}
Given $Ker(U)$ is a reducing subspace of the right-shift operator, by Theorem 2.2.1 in \cite{MR07} we have $Ker(U) = \ell^2(\mathbb{N}_0)$ or $Ker(U) = \{0\}$. Since $\{f_n\}_{n \geq 0}$ is complete, this implies $Ker(U) = \{0\}$. It follows that  $\{f_n\}_{n \geq 0}$ is a Riesz basis. 
\end{proof}
\medskip

\begin{definition}
    Let $A: \ell^2(\mathbb{N}_0) \to H^2(\mathbb{T})$, where $A(c_0, c_1, \ldots) = \underset{n \geq 0}{\sum} c_{n}z^n$.
\end{definition}
Note that the map $A$ is unitary and $ARA^* = S$ by Theorem 2.1.9 in \cite{MR07}.
\begin{lemma}
 Let $W \subset \ell^2(\mathbb{N}_0)$ be a non-trivial subspace such that $A(W)= E \subset H^2(\mathbb{T})$. Then $E$ is invariant under $S$ if and only if $W \subset \ell^2(\mathbb{N}_0)$ is invariant under $R$.
\end{lemma}

\begin{proof}
 Given $A(W)= E \subset H^2(\mathbb{T})$, $E$ is a non-trivial subspace of $H^2(\mathbb{T})$. Suppose $E$ is invariant under $S$. Then $S(E) \subset E$. That is 
 \begin{displaymath}
     S(E) = ARA^*(E) = A(R(W)) \subset E.
 \end{displaymath}
  Since $A$ is injective, if $R(W) \not\subset W$, then 
  \begin{displaymath}
      S(E) = A(R(W)) \not\subset A(W) = E.
  \end{displaymath}
   Similarly, if $W \subset \ell^2(\mathbb{N}_0)$ is invariant under $R$ then $E$ is invariant under $S$.
\end{proof}
\medskip

The following result was proved for frames indexed by $\mathbb{Z}$ in \cite{CH16}, but an effortless reworking shows that the theorem holds for frames indexed by $\mathbb{N}_0$ as well. 

\begin{theorem}
Let $\{f_n\}_{n \geq 0}$ be a frame admitting the form $\{T^{n} \varphi \}_{n \geq 0}$ for some linear operator $T: \text{\text{span}} \{f_n\}_{n \geq 0} \to \text{\text{span}} \{f_n\}_{n \geq 0}$. Then $T$ is bounded if and only if $Ker(U)$ is invariant under right-shifts.
\end{theorem}

\begin{definition}
An invariant subspace, $W$, for an operator, $\Phi$, is cyclic if there exists a vector $g$ such that $W = \overline{\text{\text{span}}} \{\Phi^{n} g : n \geq 0\}$. That is, $W$ is generated by the vectors $\{\Phi^{n} g\}_{n \geq 0}$.
\end{definition}
\begin{theorem}
Let $\{f_n\}_{n \geq 0}$ be a frame. Then the following are equivalent:

$(i.)$ $\{f_n\}_{n \geq 0}$  admits the form $\{T^{n} \varphi \}_{n \geq 0}$ with $T \in B(H)$.

\medskip
$(ii.)$  $\{f_n\}_{n \geq 0}$ is either a Riesz basis or it is linearly independent and $Ker(U)$ is generated by the vectors $\{R^{n} c\}_{n \geq 0}$ for some fixed scalar sequence $c \in \ell^2(\mathbb{N}_0)$ such that $A(c) \in H^2(\mathbb{T})$ is an inner function. 
\end{theorem}

\begin{proof}
By Proposition 2.3 in \cite{CH19}, a frame $\{f_n\}_{n \geq 0}$  admits the form $\{T^{n} \varphi \}_{n \geq 0}$ for some linear (but not necessarily bounded) operator $T$ if and only if the frame is linearly independent. By Theorem 10, the operator $T$ is bounded if and only if $Ker(U)$ is invariant under right-shifts. Thus (given the frame is linearly independent) we only need to prove equivalence of $(ii)$ and the property that $Ker(U)$ is invariant under right-shifts. Suppose $Ker(U)$ is invariant under right-shifts. Then, by Lemma 9, its image under $A$ in $H^2(\mathbb{T})$ is an invariant subspace of the operator $S$. That is, $E = A(Ker(U)) \subset H^2(\mathbb{T})$ has the property that $S(E) \subset E$. Hence, by Beurling’s Theorem \cite{B48} we have $E = \{0\}$ or $E = \psi H^2(\mathbb{T})$ for some inner function $\psi \in H^2(\mathbb{T})$. If $E = \{0\}$ then since $A$ is injective we have $Ker(U) = \{0\}$, so that $\{f_n\}_{n \geq 0}$ is a Riesz basis. Otherwise, $E = \psi H^2(\mathbb{T})$ and by Corollary 2.2.13 in \cite{MR07}, $E$ is a cyclic invariant subspace of $S$. In particular, $E = \overline{\text{span}} \{S^{n} \psi : n \geq 0\}$. Since 
\begin{displaymath}
Ker(U) = A^*(E) = A^* (\psi H^2(\mathbb{T})) =A^* (\overline{\text{span}} \{S^{n} \psi : n \geq 0\}) = \overline{\text{span}} \{A^* S^{n} \psi : n \geq 0\} 
\end{displaymath}
\begin{displaymath}
=\overline{\text{span}} \{A^* S^{n} A c : n \geq 0\} = \overline{\text{span}} \{R^{n}c : n \geq 0\}.
\end{displaymath}
In this case we have that $Ker(U)$ is is generated by the vectors $\{R^{n} c\}_{n \geq 0}$ for some fixed scalar sequence $c \in \ell^2(\mathbb{N}_0)$ such that $A(c) = \psi \in H^2(\mathbb{T})$ is an inner function. 

For the other direction, first assume that $\{f_n\}_{n \geq 0}$ is a Riesz basis. Then since $\{f_n\}_{n \geq 0}$ is $\omega$-independent, we get $Ker(U) = \{0\}$ and thus $Ker(U)$ is invariant under right-shifts. Now, if $Ker(U)$ is generated by the vectors $\{R^{n} c\}_{n \geq 0}$ for some fixed scalar sequence $c \in \ell^2(\mathbb{N}_0)$ such that $A(c) = \psi \in H^2(\mathbb{T})$ is an inner function, then 
\begin{displaymath}
A(Ker(U)) = A(\overline{\text{span}} \{R^{n}c : n \geq 0\}) = \overline{\text{span}} \{AR^{n}c : n \geq 0\} 
\end{displaymath}
\begin{displaymath}
=\overline{\text{span}} \{AR^{n}A^*\psi : n \geq 0\} = \overline{\text{span}} \{S^{n} \psi : n \geq 0\} = \psi H^2(\mathbb{T})
\end{displaymath}
(where the last equality follows since $\overline{\text{span}} \{S^{n} \psi : n \geq 0\}$ contains all functions of the form $\psi p(z)$, where $p$ is a polynomial and polynomials are dense in $H^2(\mathbb{T})$).  By Beurling  \cite{B48} we get that $\psi H^2(\mathbb{T})$ is invariant under $S$ for any inner function $\psi$. Since $Ker(U) = A^*(\psi H^2(\mathbb{T}))$ it follows that $Ker(U)$ is invariant under right-shifts. 
\end{proof}
\medskip

\begin{proposition}
Assume $\{f_n\}_{n \geq 0}$ is an overcomplete frame that admits the form $\{T^{n} f_0\}_{n \geq 0}$ with $T$ bounded. Then the system $\{R^{n} c\}_{n \geq 0}$ in Theorem 12 is a Parseval
frame for $Ker(U)$.
\end{proposition}
\begin{proof}
Let $\psi = A(c)$ as in Theorem 12. We have  $A^* S^n \psi = R^n c$ for each $n \in \mathbb{N}_0$. As $A$ is unitary, we only need to show that $\{S^n \psi\}_{n \geq 0}$ is a Parseval frame for $\psi H^2(\mathbb{T})$. Let $f \in \psi H^2(\mathbb{T})$. Then $f = \psi h$ for some $h \in H^2(\mathbb{T})$. Note that 
\begin{displaymath}
\|f\|^2 = \|\psi h\|^2 = | \langle \psi h, \psi h \rangle| = | \langle \psi \overline{\psi} h, h \rangle| = | \langle h, h \rangle| = \|h\|^2 .
\end{displaymath}

Thus as $\{z^n\}_{n \geq 0}$ is an orthonormal basis for $H^2(\mathbb{T})$, we have
\begin{displaymath}\underset{n \geq 0}{\sum} | \langle S^n \psi, f \rangle|^2 = \underset{n \geq 0}{\sum} | \langle z^n \psi, \psi h \rangle|^2 =  \underset{n \geq 0}{\sum} | \langle z^n \psi \overline{\psi}, h  \rangle|^2  = \underset{n \geq 0}{\sum} | \langle z^n, h \rangle|^2 = \|h\|^2 = \|f\|^2.
\end{displaymath}
That is, $\{S^n \psi\}_{n \geq 0}$ is a Parseval frame for $\psi H^2(\mathbb{T})$. It follows that $\{A^* S^n \psi\}_{n \geq 0} = \{R^{n} c\}_{n \geq 0}$ in Theorem 12 is a Parseval frame for $A^*(\psi H^2(\mathbb{T})) = Ker(U)$.
\end{proof}
\medskip

\begin{corollary}
Assume $\{f_n\}_{n \geq 0}$ is an overcomplete frame that admits the form $\{T^{n} f_0\}_{n \geq 0}$ with $T$ bounded. Then the system $\{R^{n} c\}_{n \geq 0}$ in Theorem 12 is an orthonormal basis for $Ker(U)$.
\end{corollary}
\begin{proof}
We have $\{z^n\}_{n \geq 0}$ is an orthonormal basis for $H^2(\mathbb{T})$. It follows that $\{\theta z^n\}_{n \geq 0} = \{S^{n} \theta : n \geq 0\}$ is an orthonormal basis for $A(Ker(U)) = \theta H^2(\mathbb{T})$ since for any $ m, n \in \mathbb{N}_0$ such that $m \neq n$ we have 
\begin{displaymath}
\langle S^n \theta, S^m \theta \rangle = \langle \theta z^n , \theta z^m  \rangle=  \langle \theta \overline{\theta }z^n , z^m \rangle = \langle z^n , z^m  \rangle = 0
\end{displaymath}
and since $\{S^{n} \theta : n \geq 0\}$ is complete in $A(Ker(U))$, we have that it is a complete orthonormal set in $A(Ker(U))$, that is an orthonormal basis for $A(Ker(U))$. Since 
$A$ is unitary and $A^* (S^n \theta) = R^n c$ for each $n$, it follows that $\{R^{n} c: n \geq 0\}$ is an orthogonal set and thus a minimal Parseval frame for $Ker(U)$. The result follows as a minimal Parseval frame is in fact an orthonormal basis.

\end{proof}

%Some of the simplest invariant subspaces of the operator $M_z$ are those comprised of functions in $H^2(\mathbb{T})$ which are the radial limits of functions in $H^2(\mathbb{D})$  with zeros on some subset of $\mathbb{D}$. For example, for any $z_0 \in \mathbb{D}$, the set $\{f \in H^2(\mathbb{D}): f(z_0) =0\}$ is an invariant subspace of $M_z$. By Beurling's Theorem, we get $\{f \in H^2(\mathbb{D}): f(z_0) =0\} = \psi H^2(\mathbb{D})$ for some inner function $\psi$. In particular, the inner function $\psi_0(z) = \frac{z_0 - z}{1 - \overline{z_0} z}$ satisfies $\{f \in H^2(\mathbb{D}): f(z_0) =0\} = \psi_0 H^2(\mathbb{D})$. Given $\{z_1, \ldots, z_k\} \subset \mathbb{D}$, the set $\{f \in H^2(\mathbb{D}): f(z_1)= \ldots = f(z_k) = 0\}$ is also an invariant subspace of $M_z$ and the inner function $\phi(z) = \underset{j=1}{\overset{k}{\prod}} \frac{z_j -z}{1 - \overline{z_j}z}$ satisfies $\{f \in H^2(\mathbb{D}): f(z_1)= \ldots = f(z_k) = 0\} = \phi H^2(\mathbb{D})$. Functions that are constant multiples of products of the form $\phi$ above are known as finite Blaschke products. 

\begin{definition}
%Let $d\ \in \mathbb{T}$, $r \in \mathbb{N}_0$, and $\{\rho_j\} \subset \mathbb{D} \setminus \{0\}$ satisfy $\underset{j=1}{\overset{\infty}{\sum}} (1-|\rho_j|)<\infty$. A Blaschke product in $H^2(\mathbb{T})$ is an inner function that is the radial limit of an function of the form $\psi(z) = dz^r\underset{j=1}{\overset{\infty}{\prod}} \frac{\overline{\rho_j}}{|\rho_j|}\cdot\frac{\rho_j -z}{1 - \overline{\rho_j}z}$. 
Given $k< \infty$, a finite Blaschke product is an inner function of the form 
\begin{displaymath}
\phi(z) = d\underset{j=1}{\overset{k}{\prod}} \frac{\lambda_j -z}{1 - \overline{\lambda_j}z} 
\end{displaymath}
where $\{\lambda_1, \ldots \lambda_k\} \subset \mathbb{D}$ and $d \in \mathbb{T}$.
\end{definition}

\begin{proposition}
Suppose an overcomplete frame, $\{f_n\}_{n \geq 0}$,  admits the form $\{T^{n} f_0\}_{n \geq 0}$ with $T$ bounded. Then the inner function, $A(c)$, in Theorem 12 is not a finite Blaschke product.
\end{proposition}
\begin{proof}
Let $\{f_n\}_{n \geq 0}$ be an overcomplete frame that  admits the form $\{T^{n} f_0\}_{n \geq 0}$ with $T$ bounded. Then by Theorem 12, we have $Ker(U)$ is generated by the vectors $\{R^{n} c\}_{n \geq 0}$ for some fixed scalar sequence $c \in \ell^2(\mathbb{N}_0)$ such that $A(c) \in H^2(\mathbb{T})$ is an inner function. If the inner function, $A(c)$, is a finite Blaschke product, then by Theorem 3.14 in \cite{RR03} the subspace $H^2(\mathbb{T}) \ominus \phi H^2(\mathbb{T}) = (\phi H^2(\mathbb{T}))^\perp$ has dimension $k < \infty$. 
\newline
Since $A(Ker(U)) = \phi H^2(\mathbb{T})$ and as $A$ is unitary, we have for any $r \in Ker(U)^\perp$ and $h \in Ker(U)$ that 
\begin{displaymath}
    0 = \langle r, h \rangle = \langle A(r), A(h) \rangle
\end{displaymath}
so that $A(r) \in (\phi H^2(\mathbb{T}))^\perp$ and thus $dim(Ker(U)^\perp)$ is finite as $A(Ker(U)^\perp) \subset (\phi H^2(\mathbb{T}))^\perp$. However, as $Ker(U)^\perp = \overline{Rng}(U^*)$ and $U^* :H \to \ell^2(\mathbb{N}_0)$ is the adjoint of the synthesis operator of the frame $\{T^{n} f_0\}_{n \geq 0}$, by Lemma 5.2.1 in \cite{Chr03}, $U^*$ is injective. Since $H$ is infinite-dimensional, this implies $\overline{Rng}(U^*)$ is infinite-dimensional as well, so that $dim(Ker(U)^\perp)$ cannot be finite, a contradiction. Thus $A(c)$ cannot be a finite Blaschke product. 
\end{proof}
\medskip
The following definitions and theorems will be important for us in the sequel.
\begin{definition}
Let $H, \, K$ be complex separable infinite-dimensional Hilbert Spaces. Given $T \in B(H)$ and $A \in B(K)$ we say the pairs $(T, f)$ and $(A, g)$ are similar and write $(T, f) \cong (A, g)$ if there exists $L \in GL(H, K)$ such that $LTL^{-1} = A$ and  $Lf = g$.
\end{definition}

\begin{definition}
A model space, $K_{\theta}$, is a subspace of $H^2(\mathbb{T})$ of the form $K_{\theta} = H^2(\mathbb{T}) \ominus \theta H^2(\mathbb{T})$ for some shift-invariant subspace $\theta H^2(\mathbb{T})$.
\end{definition}

\begin{definition}
Let $K_{\theta}$ be a model space. The map $S_{\theta} = P_{K_{\theta}} \left.S\right|_{K_{\theta}}$
is the compression of the shift to $K_{\theta}$ where $P_{K_{\theta}}$ is the orthogonal projection onto $K_{\theta}$.
\end{definition}

\begin{theorem*}[Theorem 3.14 in \cite{RR03}]
A model space $K_{\theta} = H^2(\mathbb{T}) \ominus \theta H^2(\mathbb{T})$ is of finite dimension if and only if the inner function $\theta$ is a finite Blaschke product. 
\end{theorem*}
 
In Theorem 12 above we have a characterization of the frames that can be represented as iterations of an operator on a single vector. That is, in Theorem 12 we were given a frame a priori and conditions were provided so that the given frame had the form $\{T^nf_0\}_{n \geq 0}$. However, the following theorem found in \cite{CHP20} gives a necessary and sufficient condition for iterations of a bounded operator on a single vector to generate a frame for an separable infinite dimensional Hilbert space. This theorem characterizes the systems $\{T^nf_0\}_{n \geq 0}$ that form a frame for $H$. In this theorem, the aim is not to provide conditions for rewriting a given frame as iterations of a bounded operator on a vector, but to give a solution to the Dynamical Sampling problem as posed in \cite{AA16}.
\begin{theorem*}[Theorem 3.6 in \cite{CHP20}]
Let $T \in B(H)$ and $f_0 \in H$. Then the following hold:

$(i.)$ $\{T^nf_0\}_{n \geq 0}$ is an overcomplete frame for $H$ if and only if $(T, f_0)\cong (S_{\theta}, P_{K_\theta} {1}_{\mathbb{T}})$ where $\theta (z) \in H^2(\mathbb{T})$ is an inner function that is not a finite Blashke product.

\medskip
$(ii.)$ $\{T^nf_0\}_{n \geq 0}$ is a Riesz basis for $H$ if and only if  $(T, f_0)\cong (S, {1}_{\mathbb{T}})$. 

\end{theorem*}
\begin{proposition}
Let $\{T^nf_0\}_{n \geq 0}$ be an overcomplete frame for $H$. The inner function in Theorem 3.6 of \cite{CHP20} and the one given in Theorem 12 differ at most by a unimodular constant factor.
\end{proposition}
\begin{proof}
%By Proposition 2.3 in \cite{CH19}, given a frame has the form $\{T^nf_0\}_{n \in \mathbb{N}_0}$, the frame is linearly independent. Thus %due to the identification of  $H^2(\mathbb{T})$ with $H^2(\mathbb{D})$ 
 Let $c \in Ker(U)$ be the scalar sequence given in Theorem 12 such that $Ac = \psi$ with $A(Ker(U)) = \psi  H^2(\mathbb{T})$. Let  $\mathcal{F}: H^2(\mathbb{T}) \to \ell^2(\mathbb{N}_0)$ be the Fourier Transform and let $V = U \mathcal{F}$. We aim to show that for the shift-invariant subspace $Ker(V) = \phi H^2(\mathbb{T})$, for some inner function $\phi$, we have $\phi H^2(\mathbb{T}) = \psi  H^2(\mathbb{T})$ so that by Theorem 3.9 in \cite{RR03}, the inner functions $\phi$ (given as $h$ in Theorem 3.6 of \cite{CHP20}) and $\psi$ differ by a unimodular constant factor. 
 
 Let $f \in Ker(V)$. Since $f \in H^2(\mathbb{T})$ we have $f = \underset{n \geq 0}{\sum}a_n z^n$ for some 
 \newline
 $\{a_n\}_{n \in \mathbb{N}_0} \in \ell^2(\mathbb{N}_0)$. Note that $\mathcal{F} f = \{a_n\}_{n \in \mathbb{N}_0}$ so that as 
 $0 = Vf = U \mathcal{F} f = U \{a_n\}_{n \in \mathbb{N}_0}$,
 we get $\{a_n\}_{n \in \mathbb{N}_0} \in Ker(U)$. Since $f = \underset{n \geq 0}{\sum}a_n z^n = A \{a_n\}_{n \in \mathbb{N}_0}$ we have $ f \in A(Ker(U))$. That is $ \phi  H^2(\mathbb{T}) \subset \psi  H^2(\mathbb{T})$.
 
Now, let $f \in A(Ker(U))$. Then $f = \underset{n \geq 0}{\sum}c_n z^n $ for some $\{c_n\}_{n \in \mathbb{N}_0} \in Ker(U)$. Thus, $Vf = U \mathcal{F} f = U (\{c_n\}_{n \in \mathbb{N}_0}) = 0$. That is, $f \in Ker(V) $ so that $ \psi  H^2(\mathbb{T}) \subset \phi  H^2(\mathbb{T})$. Hence $\phi H^2(\mathbb{T}) =  \psi H^2(\mathbb{T})$. By Theorem 3.9 in \cite{RR03} it follows that $\phi$ and $\psi$ differ at most by a unimodular constant factor.
\end{proof}
\medskip
In Theorem 3.6 of \cite{CHP20} when the frame $\{T^n f_0\}_{n \geq 0}$ is a Riesz basis, the function $\phi$ ($h$ in their notation) is 0. Similarly in Theorem 12, since $Ker(U) = \{0\}$ in this case, the sequence is zero so that $\psi = Ac = 0$. Thus the functions are equal in this case.
\newline

In \cite{CHP20} the authors pose the question of whether it is possible, for a given $T \in B(H)$, to characterize the set of vectors $f_0$ such that $\{T^n f_0 \}_{n \geq 0}$ is a frame. In Proposition 3.10 of \cite{CHP20}, the set, $\mathcal{G}_T = \{f_0 \in H : \{T^n f_0 \}_{n \geq 0}$ is a frame\}, of vectors that can be used to generate a frame via iterations of $T$ is characterized in the following way. Given a vector $f_0$ such that $\{T^n f_0 \}_{n \geq 0}$ is a frame, $\mathcal{G}_T = \{Vf_0: V \in \{T\}' \cap GL(H)\}$. That is the set $\mathcal{G}_T$ consists of all the images of $f_0$ under a bounded invertible operator in the commutant of $T$. Since by Theorem 3.6, $\{T^n f_0 \}_{n \geq 0}$ is a Riesz basis if and only if $(T, f_0)\cong (S, {1}_{\mathbb{T}})$ so that $T \sim S$, and $\{T^n f_0 \}_{n \geq 0}$ is an overcomplete frame if and only if $(T, f_0)\cong (S_{\theta}, P_{K_\theta} {1}_{\mathbb{T}})$ so that $T \sim S_\theta$ (where $S_\theta$ is the compression of the shift operator to an infinite dimensional model space $K_\theta$), we can explicitly give those operators in $\{T\}' \cap GL(H)$ as precisely the operators that are similar to an invertible analytic Toeplitz operator on $H^2(\mathbb{T})$.

A characterization of the operators that commute with the unilateral shift, $S \in B(H^2(\mathbb{T}))$, due to  Brown and Halmos \cite{B64}, showed that the commutant of the shift,  $\{S\}', \, $consists solely of the analytic Toeplitz operators on $H^2(\mathbb{T})$. 
%Toeplitz operators are, in a sense, compressions of multiplication operators on to $H^2(\mathbb{T})$.

\begin{definition}
For every $\psi \in L^{\infty}(\mathbb{T})$, the Toeplitz operator on $H^2(\mathbb{T})$ with symbol $\psi$ is the operator $M_{\psi}$ defined by 
\begin{displaymath}
M_{\psi} f(z) = P_{H^2(\mathbb{T})} \psi(z) f(z), \, \, \, \forall f \in H^2(\mathbb{T}).
\end{displaymath}
\end{definition}
\begin{definition}
A Toeplitz operator,\,$M_{\psi}$\,, is an analytic Toeplitz operator if
\begin{displaymath}
    \psi \in H^2(\mathbb{T}) \cap L^{\infty}(\mathbb{T}) = H^{\infty}(\mathbb{T}).
\end{displaymath}
\end{definition}
%\begin{theorem}[3.3.8 Operators Hardy Hilbert Rosenthal]
%If $\psi \in H^{\infty}$, then $\sigma(M_\psi)$ is the closure of $\psi(\mathbb{D})$
%\end{theorem}
\begin{theorem*}[Brown and Halmos]
An operator $\Omega \in B(H^2(\mathbb{T}))$ satisfies $\Omega S = S \Omega$ if and only if $ \Omega = M_\psi$ for some $\psi \in H^{\infty}(\mathbb{T})$.
\end{theorem*}
This result due to Brown and Halmos shows that the commutant of the shift operator, $\{S\}' \, $, on $H^2(\mathbb{T})$ is comprised of all the analytic Toeplitz operators on $H^2(\mathbb{T})$.

Due to the identification of functions in $H^2(\mathbb{T})$ and $H^2(\mathbb{D})$, regarding $\psi \in H^2(\mathbb{T})$ as it corresponding function in $H^2(\mathbb{D})$, we can make use of the following theorem.

\begin{theorem*}[3.3.8 in \cite{MR07}]
If $\psi \in H^{\infty}(\mathbb{D})$, then the spectrum, $\sigma(M_\psi)$, of the analytic Toeplitz operator $M_\psi$ is the closure of $\psi(\mathbb{D})$.
\end{theorem*}
From this we see that for a given function $\psi \in H^{\infty}(\mathbb{D})$, as long as $0$ is not in the closure of the range of the function, the associated analyitic Toeplitz operator $M_\psi$ will be invertible. 

\begin{corollary}
Let $\psi \in H^{\infty}(\mathbb{D})$. The analytic Toeplitz operator $M_\psi$ is invertible if and only if $\psi \in H^{\infty}(\mathbb{D})$ is bounded away from $0$.
\end{corollary}
\begin{proof}
The result follows immediately from the theorem above.
\end{proof}

\begin{proposition}
Let $\{T^n f_0 \}_{n \geq 0}$ be a Riesz basis for $H$ so that $T = LSL^{-1}$ for some $L \in GL(H^2(\mathbb{T}), H)$. Then $\mathcal{G}_T = \{Af_0: A = LM_\varphi L^{-1}$ where $\varphi$ is the radial limit of a function in $H^{\infty}(\mathbb{D})$ which is bounded away from $0$\}.
\end{proposition}
\begin{proof}
The commutant of the shift operator, $S$, is the set of all analytic Toeplitz operators on $H^2(\mathbb{T})$. Since $T = LSL^{-1}$ and when $M_\varphi$ is invertible, so is $LM_\varphi L^{-1}$ we have for any operator $A = LM_\varphi L^{-1}$,
\begin{displaymath}
    TA = LSL^{-1} LM_\varphi L^{-1} = LS M_\varphi L^{-1} = L M_\varphi SL^{-1} = L M_\varphi L^{-1}LSL^{-1} = AT
\end{displaymath}
 That is $A \in \{T\}' \cap GL(H)$. For the reverse containment, let $A \in \{T\}' \cap GL(H)$. That is $AT = TA$ where $A$ is invertible. Since $L^{-1}TL = S$, we have  
 \begin{displaymath}
     L^{-1}ALS= L^{-1}ATL =  L^{-1}TAL = S L^{-1}AL
 \end{displaymath}
  so that $L^{-1}AL$ is in the commutant of $S$. This implies that $L^{-1}AL$ is an invertible analytic Toeplitz operator on $H^2(\mathbb{T})$. That is $\mathcal{G}_T = \{Af_0: A = LM_\varphi L^{-1}$ where $\varphi$ is the radial limit of a function in $H^{\infty}(\mathbb{D})$ which is bounded away from $0$\}.
\end{proof}
\medskip
In order to characterize the set $\mathcal{G}_T$ in the case where $\{T^n f_0 \}_{n \geq 0}$ is an overcomplete frame, we need to use the fact that $T \sim S_\theta$ where $S_\theta = P_{K_\theta}\left.S\right|_{K_{\theta}}$. That is, $T$ is similar to a compression of the unilateral shift to some model space $K_{\theta}$. 

\begin{definition}
Let $\theta$ be an inner function and ${K_{\theta}}$ be the associated model space. A truncated analytic Toeplitz operator is a compressed analytic Toeplitz operator. That is, for any analytic Toeplitz operator, $M_\psi$, its compression to the model space
\begin{displaymath}
    M_{\psi}^{\theta} = P_{K_\theta}\left.M_\psi\right|_{K_{\theta}}
\end{displaymath}
is a truncated analytic Toeplitz operator.
\end{definition}

The commutant of the compression of the shift for some inner function, $\{S_\theta\}'$, was described by Sarason in \cite{S67}. This characterization of the operators in $\{S_\theta\}'$ by Sarason is known as the Commutant Lifting Theorem for the Compression of the Shift \cite{GMR16}. The result shows that the commutant of the compression of the shift is precisely the set of truncated analytic Toeplitz operators.
\begin{theorem*}[Sarason]
Let $\theta$ be a nonconstant inner function. Let $S_\theta = P_{K_\theta}\left.S\right|_{K_{\theta}}$, that is, a compression of the unilateral shift to some model space $K_{\theta}$. An operator $\Omega \in B(K_{\theta})$ satisfies $\Omega S_\theta = S_\theta \Omega$ if and only if $\Omega = M_{\psi}^{\theta}$ for some $\psi \in H^{\infty}(\mathbb{T})$.
\end{theorem*}
A characterization of the invertible truncated analytic Toeplitz operators can be provided by applying the following result due to Fuhrmann \cite{F68}.

\begin{theorem*}[Fuhrmann]
Let $\theta$ be an inner function and $\psi \in H^{\infty}(\mathbb{D})$. Then
\begin{displaymath}
\sigma(M_{\psi}^{\theta}) = \{\lambda \in \mathbb{C} : \underset{z \in \mathbb{D}}{\text{inf}} \, \,  (|\theta(z)| + |\psi(z) - \lambda|) = 0\}.
\end{displaymath}
\end{theorem*}
\begin{corollary}
    A truncated analytic Toeplitz operator, $M_{\psi}^{\theta}$, is invertible if and only if $\theta$ and $\psi$ satisfy $\underset{z \in \mathbb{D}}{\text{inf}} \, \,  (|\theta(z)| + |\psi(z)|) > 0$.
\end{corollary}
\begin{proof}
The result follows immediately from the theorem above.
\end{proof}

We note that it is observed in Remark 3.11 in \cite{CHP20} that when $\{T^n f_0 \}_{n \geq 0}$ is an overcomplete frame, a similar characterization as the one provided in the proposition below can be given for the set $\mathcal{G}_T$. 
\begin{proposition}
Let $\{T^n f_0 \}_{n \geq 0}$ be an overcomplete frame for $H$ so that $T = L S_\theta L^{-1}$ for some $L \in GL(K_{\theta}, H)$. Then
\begin{displaymath}
\mathcal{G}_T = \{Af_0: A = LM_{\psi}^{\theta} L^{-1} \, \,  \text{where} \, \,  \theta \, \,  \text{and} \, \,   \psi \, \, \text{satisfy} \, \,  \underset{z \in \mathbb{D}}{\text{inf}} \, \,  (|\theta(z)| + |\psi(z)|) > 0\}. 
\end{displaymath}
\end{proposition}
\begin{proof}
    Applying the fact that $T \sim S_\theta$ and the theorem above due to Sarason, the result follows similarly to the proof of Proposition 27.
\end{proof}

\section
{Auxiliary Results on Operator Representations of Riesz Frames}
%\subsection{Operator Representations of Riesz Frames}
\begin{definition}
A frame $\{f_n\}_{n \in I} \subset H$ with the property that for every $J \subset I$ we have  $\{f_{n}\}_{n \in J}$ is a frame sequence with uniform frame bounds $C_1$ and $C_2$ is said to be a Riesz frame.
\end{definition}

Note that not every frame exhibits %the subframe property (satisfies the condition 
the property that every subsequence is a frame sequence (i.e. the subframe property), however every subsequence of a Riesz basis is a Riesz sequence.

%\begin{definition}
%Given $\{e_n\}_{n \in I} \subset H$ is an orthonormal basis, a Riesz basis is any family of the form $\{Ae_n\}_{n \in I}$ where $A \in B(H)$ is a bijection.
%\end{definition}

By Theorem 7.13 in \cite{Hei11}, a sequence $\{f_n\}_{n \in I} \subset H$ is a Riesz basis if and only if it is a bounded and unconditional basis for $H$. By Theorem 7.4.3 in \cite{Chr03}, every Riesz frame contains a Riesz basis. However, Riesz frames are distinct from near-Riesz bases as detailed in \cite{Chr03}. 

\begin{theorem}
An overcomplete Riesz frame in $H$ cannot admit the form $\{T^{n} \varphi \}_{n \geq 0}$ for any linear operator $T$. 
\end{theorem}
\begin{proof}
Every linearly independent Riesz frame is a Riesz basis by Lemma 2.2 in \cite{CA02}. Thus, an overcomplete Riesz frame must be a linearly dependent set. By Proposition 2.3 in \cite{CH19}, the result follows.  
\end{proof}

\begin{corollary}
A frame of the form $\{T^{n} \varphi \}_{n \geq 0}$ cannot be a Riesz frame if $T$ is a normal operator.
\end{corollary}

\begin{proof}
 If $\{T^{n} \varphi\}_{n \geq 0}$ is a Riesz frame, then the system must be minimal by Theorem 32. That is, if $\{T^{n} \varphi\}_{n \geq 0}$ is a Riesz frame, then the frame is exact. However, when $T$ is normal, by Proposition 4.1 in \cite{ACM16}, a system of the form $\{T^{n} \varphi\}_{n \geq 0}$ can never be exact.
\end{proof}

\begin{remark*}
It was proved in \cite{ACM16} that for a normal operator $T$, a system $\{T^{n} \varphi \}_{n \geq 0}$ is a frame if and only if $T = \underset{j \in I}{\sum} \lambda_j P_j$ where $P_j$ is a rank one orthogonal projection for each $j$ and $\{\lambda_k\}_{k \in I} \subset \mathbb{D}$ is a uniformly separated sequence such that $|\lambda_k| \to 1$ and $C_1 \leq \frac{\|P_j \varphi \|}{\sqrt{1-|\lambda_k|^2}} \leq C_2$, for some positive constants $C_1$ and $C_2$. By Corollary 37, frames generated by orbits of operators of this form can never admit the subframe property with uniform frame bounds.
\end{remark*}
\begin{proposition}
Every Riesz frame can be represented as a finite union of the form 
\newline 
$\underset{1 \leq i \leq N} \bigcup \{T_{i}^{n} \varphi_i\}_{n \geq 0}$ with $T_i \in B(H)$ for each $i \in \{1, \ldots, N\}$. 
\end{proposition}

\begin{proof}
Every Riesz frame can be represented as a finite union of Riesz sequences by Corollary 2.5 in \cite{CA02}. Each Riesz sequence in the finite union is bounded, so a Riesz frame must be  norm bounded below. By Theorem 2.1 in \cite{CH19}, the result follows.
\end{proof}
%\begin{definition}
%{f_n\}_{n \in \mathbb{N}_0}$, is the map $U: \ell^2(\mathbb{N}_0) \to H$ such that $U \{c_n\}_{n \in \mathbb{N}_0} = \underset{n \geq 0}{\sum} c_n f_n$.
%\end{definition}

\section{Appendix}
 \begin{lemma*}[Lemma 2.2 in \cite{CA02}]
Let $\{f_k\}_{k \in I}$ be a frame for $H$. Let $\{I_n\}_{n \in \mathbb{N}}$ be a family of finite subsets of $I$ such that 
\begin{displaymath}
    I_1 \subset I_2 \subset I_3 \uparrow I.
\end{displaymath}
Let $A_n$ denote the optimal lower frame bound for $\{f_k\}_{k \in I}$ as a frame for its span. Then the following are equivalent:
\newline
$(i.)$ $\{f_k\}_{k \in I}$ is a Riesz basis.
\newline
$(ii.)$ $\{f_k\}_{k \in I}$ is $\omega$\text{-}independent.
\newline
$(iii.)$ $\{f_k\}_{k \in I}$ is linearly independent and $inf_{n} A_n >0 $
\end{lemma*}

\begin{proposition*}[Proposition 2.3 in \cite{CH19}]
Consider any sequence $\{f_k\}_{k=1}^{\infty}$ in $H$ for which $\text{\text{span}} \{f_k\}_{k=1}^{\infty}$ is infinite-dimensional. Then the following are equivalent: 
\newline
$(i.)$ $\{f_k\}_{k=1}^{\infty}$ is linearly independent. 
\newline
$(ii.)$ There exists a linear operator $T: \text{\text{span}} \{f_k\}_{k=1}^{\infty} \to H$ such that $\{f_k\}_{k=1}^{\infty} = \{T^n f_1\}_{n=0}^{\infty}$.
\end{proposition*}

\begin{proposition*} [Proposition 4.1 in \cite{ACM16}]
If $A$ is a normal operator on $H$, then for any set of vectors $\mathcal{G} \subset H$, the system of iterates $\{A^n g\}_{{g \in \mathcal{G}},  n \geq 0}$ is not a complete and minimal system in $H$.
\end{proposition*}

\begin{theorem*}[Theorem 2.1 in in \cite{CH19}]
Consider a frame $\{f_k\}_{k=1}^{\infty}$ which is norm-bounded below. Then the following hold:
\newline
$(i.)$ The frame $\{f_k\}_{k=1}^{\infty}$ can be decomposed as a finite union 
\begin{displaymath}
    \{f_k\}_{k=1}^{\infty} = \underset{1 \leq j \leq J} \bigcup \{f_{k}^{(j)} \}_{k \in I_j}
\end{displaymath}
where each of the sequences $\{f_{k}^{(j)} \}_{k \in I_j}$ is an infinite Riesz sequence. 
\newline
$(ii.)$  There is a finite collection of vectors from $\{f_k\}_{k=1}^{\infty}$ to be called $\varphi_1 \ldots \varphi_J$, and corresponding bounded operators $T_j: H \to H$ with closed range, such that
\begin{displaymath}
    \{f_k\}_{k=1}^{\infty} = \underset{1 \leq j \leq J} \bigcup \{T_{j}^{n} \varphi_j\}_{n \geq 0}
\end{displaymath}
\end{theorem*}

%\begin{theorem}[Theorem 2.1.9 in \cite{MR07}]
%The shift operator $S$ on $H^2(\mathbb{T})$ is unitarily equivalent to $R$.
%\end{theorem}

\begin{theorem*}[Theorem 2.2.1 in \cite{MR07}]
The only reducing subspaces of $R$ are $\{ 0 \}$ and the entire space. 
\end{theorem*}

\begin{theorem*}[Beurling's Theorem]
  Every nontrivial invariant subspace of the shift operator $S \in B(H^2(\mathbb{T})) $, where $Sf(z) = zf(z)$, is of the form $\theta H^2(\mathbb{T})$ for some inner function $\theta$. Conversely, for any inner function $\theta$, the subspace $\theta H^2(\mathbb{T})$ is invariant under $S$.
\end{theorem*}

\begin{corollary*}[Cor 2.2.13 in \cite{MR07}]
Every invariant subspace of $S$ is cyclic.
\end{corollary*}

%\begin{theorem*}[Theorem 3.14 in \cite{RR03}]
%If $\{\lambda_1, \ldots \lambda_k\} \subset \mathbb{D}$ and if 
%\begin{displaymath}
%\phi(z) = \underset{j=1}{\overset{k}{\prod}} \frac{\lambda_j -z}{1 - \overline{\lambda_j}z} 
%\end{displaymath}
%then $H^2(\mathbb{T}) \ominus \phi H^2(\mathbb{T})$ has dimension k. Conversely every shift-invariant subspace of codimension k has this form. 
%\end{theorem*}

%\begin{theorem}[Theorem 3.6 in \cite{CHP20}]
%A system $\{T^{n} \varphi \}_{n \geq 0} \subset H$, with $T \in B(H)$ is an overcomplete frame if and only if  $(T, \varphi) \cong (S_{\theta}, P_{K_\theta}1_\mathbb{T})$ where $\theta$ is some unique inner function that is not a finite Blaschke product.
   %\end{theorem}

\section
{Conflict of Interest Statement}
The corresponding author states that there is no conflict of interest. 

\end{document}